\documentclass[preprint,12pt]{elsarticle}

\usepackage{graphicx}
\usepackage[margin=1.0in]{geometry} 

\usepackage{amsmath} 
\usepackage{amssymb}
\usepackage{mathrsfs}
\usepackage{amsthm}
\usepackage{MnSymbol}

 \numberwithin{dummy}{section} 
\newtheorem{thm}{Rule} 
\newtheorem{thm2}{Restriction} 

\newdefinition{rmk}{Remark}

\usepackage{hyperref}
\usepackage[usenames, dvipsnames]{color}

\usepackage{fancyhdr}
\usepackage{lipsum} 
\usepackage[linesnumbered,ruled,vlined]{algorithm2e}

\journal{arXiv}

\begin{document}

\begin{frontmatter}

\title{An Efficient Algorithm to Test Potentially Bipartiteness of Graphical Degree Sequences}


\author{Kai Wang\corref{corkw}}
\ead{kwang@georgiasouthern.edu}
\address{Department of Computer Science, Georgia Southern University, Statesboro, GA 30460}

\cortext[corkw]{Corresponding author}

\begin{abstract}
As a partial answer to a question of Rao, a deterministic and customizable efficient algorithm is presented
to test whether an arbitrary graphical degree sequence has a bipartite realization.
The algorithm can be configured to run in polynomial time, at the expense of possibly producing an erroneous output on
some ``yes'' instances but with very low error rate.
\end{abstract}


\end{frontmatter}

\section{Introduction}
\label{sec:Intro}
Given an arbitrary graphical degree sequence $\mathbf{d}$, let $\mathscr{R}(\mathbf{d})$ denote the set of 
all of its non-isomorphic realizations. As usual, let $\chi(G)$ and $\omega(G)$ denote the chromatic number and
clique number of a finite simple undirected graph $G$ respectively.
It is known from Punnim \cite{Punnim2002A} that for any given
$\mathbf{d}$ the set $\{\chi(G): G\in \mathscr{R}(\mathbf{d})\}$ is exactly a set of integers in some
interval. Define $X(\mathbf{d})$ to be $\max\{\chi(G): G\in \mathscr{R}(\mathbf{d})\}$
and $\chi(\mathbf{d})$ to be $\min\{\chi(G): G\in \mathscr{R}(\mathbf{d})\}$. These two quantities can be
interesting for the structural properties of all the graphs in $\mathscr{R}(\mathbf{d})$.

Good lower and upper bounds on $X(\mathbf{d})$ are known from Dvo\v{r}\'{a}k and Mohar \cite{Dvorak2013}
in terms of $\Omega(\mathbf{d})=\max\{\omega(G): G\in \mathscr{R}(\mathbf{d})\}$, which can be easily computed for any given
$\mathbf{d}$ using the algorithm from Yin \cite{Yin2012}. For example, $X(\mathbf{d})\ge \Omega(\mathbf{d})$,
$X(\mathbf{d})\le \frac{4}{5}\Omega(\mathbf{d})+\frac{1}{5}\max(\mathbf{d})+1$
and $X(\mathbf{d})\le \frac{6}{5}\Omega(\mathbf{d})+\frac{3}{5}$.

It appears computationally intractable to compute $\chi(\mathbf{d})$ for any given zero-free $\mathbf{d}$.
In this paper we are concerned with the related, somewhat easier, decision problem of whether $\chi(\mathbf{d})=2$. Clearly,
this is equivalent to decide whether $\mathbf{d}$ has a bipartite realization, which is actually the first listed unsolved
problem in Rao \cite{Rao1981} to characterize potentially bipartite graphical degree sequences and which remains
unsolved to our knowledge. Note that the input
$\mathbf{d}$ is a single sequence of vertex degrees. A related problem is to decide, given two sequences of positive
integers $(a_1,a_2,\cdots,a_m;b_1,b_2,\cdots,b_n)$, where $a_1\ge a_2
\ge \cdots \ge a_m$ and $b_1\ge b_2 \ge \cdots \ge b_n$ and $\sum_{i=1}^{m}a_i=\sum_{i=1}^{n}b_i$,
whether there is a bipartite graph whose two partite sets have $\mathbf{a}=(a_1,a_2,\cdots,a_m)$ and
$\mathbf{b}=(b_1,b_2,\cdots,b_n)$ as their respective degree sequences. This problem can be easily solved by applying
the Gale-Ryser theorem \cite{Gale1957,Ryser1957}, which states that the answer is ``yes'' if and only if the conjugate
of $\mathbf{a}$ dominates $\mathbf{b}$ (or, equivalently, the conjugate of $\mathbf{b}$ dominates $\mathbf{a}$).
Here we use the common definition of domination between two partitions of the same integer:
a partition $\mathbf{p}=(p_1,p_2,\cdots)$
dominates a partition $\mathbf{q}=(q_1,q_2,\cdots)$ if $\sum_{i=1}^j p_i\ge \sum_{i=1}^j q_i$ for each $j=1,2,\cdots$.
By convention, $p_j=0$ for $j>\ell(\mathbf{p})$, where $\ell(\mathbf{p})$ denotes the number of parts in the partition
$\mathbf{p}$. We also use $|\mathbf{p}|$ to denote the weight of the partition
$\mathbf{p}$, that is, the sum of all the parts of $\mathbf{p}$.

The rest of the paper is organized as follows. Section \ref{sec:code} describes the algorithm to decide whether
a given $\mathbf{d}$ has a bipartite realization. Section \ref{sec:complexity} gives a time complexity analysis
of the algorithm. Section \ref{sec:exp} presents some experimental results. Section \ref{sec:diss} discusses
alternative designs of the algorithm and comments on the complexity of the decision problem.
Section \ref{sec:conclusion} concludes with further research directions.

\section{Description of the Algorithm}
\label{sec:code}
Clearly, to decide whether any zero-free graphical degree sequence $\mathbf{d}=(d_1\ge d_2 \ge \cdots \ge d_n)$
with weight $|\mathbf{d}|=\sum_{i=1}^n d_i$ has a bipartite realization, we first need to determine whether
it has a bipartition into $\mathbf{a}$ and $\mathbf{b}$ of equal weights $|\mathbf{d}|/2$ (for convenience, we call such
bipartitions of $\mathbf{d}$ \textit{candidate} bipartitions).
One may feel it challenging to find a candidate bipartition of $\mathbf{d}$ in the first place,
because it looks exactly like the well-known subset sum problem, which is known to be NP-complete \cite{GareyJohnson1979}.
Fortunately, since every term in $\mathbf{d}$ of length $n$ is less than $n$, this restricted subset sum problem can be solved
easily through dynamic programming in polynomial time \cite{GareyJohnson1979,Koiliaris2019}. In fact, many
inputs admit a large number of candidate bipartitions. Now we can see that the decision problem boils down to
checking whether $\mathbf{d}$ has at least one candidate bipartition and, if this is the case, whether any of those candidate
bipartitions satisfies the Gale-Ryser condition.

A naive algorithm can simply enumerate all candidate bipartitions of $\mathbf{d}$ and check each of them against the
Gale-Ryser condition. Such an algorithm necessarily runs in exponential time in the worst case. Our algorithm is more
sophisticated than that. It has two phases. The first phase utilizes up to seven rules that can all be easily checked.
As a matter of fact, in Section \ref{sec:exp} we will show that most of the inputs can be resolved by this phase alone.
The second is the enumeration phase, in which we do ``brute-force'' search in a clever way.

In describing and justifying the seven rules in the first phase,
we seek a candidate bipartition of $\mathbf{d}$ into the left side $\mathbf{a}$
and the right side $\mathbf{b}$ in such a way that at least half of the largest terms in $\mathbf{d}$ appear in $\mathbf{a}$, without
loss of generality. For example, for any input $\mathbf{d}$ of length 50 with the largest term 34 whose multiplicity is 5 (i.e. there
are exactly 5 copies of 34 in $\mathbf{d}$),
we will seek a candidate bipartition such that the left side $\mathbf{a}$ contains at least 3 copies of 34.
\begin{thm}
	\label{thm_bipartition}
	If $\mathbf{d}$ does not have a candidate bipartition, then it is not potentially bipartite. 
\end{thm}
\begin{proof}
This rule is obvious.
As mentioned above, this rule can be easily implemented through dynamic programming for the subset sum problem.
\end{proof}

\begin{thm}
	\label{thm_Mantel}
	If $|\mathbf{d}|>\frac{n^2}{2}$, then $\mathbf{d}$ is not potentially bipartite. 
\end{thm}
\begin{proof}
Based on Mantel's theorem \cite{Mantel1907}, any simple undirected bipartite graph on $n$ vertices has at most
$\frac{n^2}{4}$ edges. So the degree sum
cannot exceed $\frac{n^2}{2}$ for any $\mathbf{d}$ that is potentially bipartite.
\end{proof}

\begin{thm}
	\label{thm_Small_Degrees}
	If $d_1+d_{n+1-d_1}>n$, then $\mathbf{d}$ is not potentially bipartite. 
\end{thm}
\begin{proof}
Suppose $\mathbf{d}$ is potentially bipartite.
The left partite set contains a vertex $v_1$ of degree $d_1$ so the right partite set contains at least $d_1$ vertices ($v_1's$
neighbors), each of which has a degree at most $n-d_1$ since the left partite set has at most $n-d_1$ vertices.
Consequently, $\mathbf{d}$ must contain at least $d_1$ degrees that are $\le n-d_1$. Therefore, $d_{n+1-d_1}$
must be $\le n-d_1$ for $\mathbf{d}$ to be potentially bipartite.
\end{proof}

\begin{thm}
	\label{thm_Large_Degrees}
	If $\sum_{i=1}^{n-d_1} d_i<\frac{|\mathbf{d}|}{2}$, then $\mathbf{d}$ is not potentially bipartite. 
\end{thm}
\begin{proof}
As mentioned in the proof of Rule \ref{thm_Small_Degrees}, the left partite set has at most $n-d_1$ vertices.
Clearly, the degree sum $|\mathbf{a}|$ of the left side $\mathbf{a}$ is impossible
to exceed $\sum_{i=1}^{n-d_1} d_i$. Therefore, $\sum_{i=1}^{n-d_1} d_i$ must be at least
$\frac{|\mathbf{d}|}{2}$ for $\mathbf{d}$ to be potentially bipartite.
\end{proof}

\begin{thm}
	\label{thm_Fixed_Degrees}
	If $\sum_{d_i>n-d_1} d_i>\frac{|\mathbf{d}|}{2}$, then $\mathbf{d}$ is not potentially bipartite. 
\end{thm}
\begin{proof}
As shown in the proof of Rule \ref{thm_Small_Degrees}, each of the right side degrees in $\mathbf{b}$ is
at most $n-d_1$. Therefore, every degree larger than $n-d_1$ must be in the left side $\mathbf{a}$ and the sum of such
degrees should not exceed $\frac{|\mathbf{d}|}{2}$ for any $\mathbf{d}$ that is potentially bipartite.
\end{proof}

For the following rule, we will need the concept of \textit{residue} of a finite simple undirected graph $G$
or a graphical degree sequence $\mathbf{d}$ introduced in Favaron et al. \cite{Favaron1991} and we use $R(G)$
and $R(\mathbf{d})$ as notations. We also use $\overline{\mathbf{d}}$ to denote the complementary 
graphical degree sequence of $\mathbf{d}$: $(n-1-d_n\ge n-1-d_{n-1}\ge \cdots \ge n-1-d_1)$, which is the degree
sequence of the complementary graph of any realization of $\mathbf{d}$.

\begin{thm}
	\label{thm_Residue}
	If $R(\overline{\mathbf{d}})\ge 3$, then $\mathbf{d}$ is not potentially bipartite. 
\end{thm}
\begin{proof}
As proved in \cite{Favaron1991}, the residue $R(\mathbf{d})$ of a graphical degree sequence $\mathbf{d}$ is a lower bound
on the independence number of any realization of $\mathbf{d}$. Then clearly $R(\overline{\mathbf{d}})$ is a lower bound
on the clique number of any realization of $\mathbf{d}$. The result follows because any graph with a clique of size
at least 3 is not bipartite.
\end{proof}

The following is a similar rule that uses the concept of Murphy's bound introduced in Murphy \cite{MURPHY1991},
denoted $\beta(G)$ or $\beta(\mathbf{d})$ here, which is also a lower bound on the independence number of any
realization of $\mathbf{d}$.

\begin{thm}
	\label{thm_Murphy}
	If $\beta(\overline{\mathbf{d}})\ge 3$, then $\mathbf{d}$ is not potentially bipartite. 
\end{thm}

If the input $\mathbf{d}$ passes the tests of all of the above seven rules and cannot be resolved as a ``no''
instance, then our algorithm will enter the enumeration phase. From Rule \ref{thm_Fixed_Degrees}
we know that by now we must have $\frac{|\mathbf{d}|}{2}\ge \sum_{d_i>n-d_1} d_i$.
In the special case that equality holds, which means the left side $\mathbf{a}$ must
contain exactly those degrees that are larger than $n-d_1$ should $\mathbf{d}$ be potentially bipartite, our algorithm
can immediately stop based on the result of the Galy-Ryser conditional test on this candidate bipartition of $\mathbf{d}$.
Otherwise, our algorithm continues with $S=\frac{|\mathbf{d}|}{2}- \sum_{d_i>n-d_1} d_i>0$, which is the sum of the additional
degrees that need to be in the left side $\mathbf{a}$ besides those that are larger than $n-d_1$.
For convenience, we use $\mathbf{a_f}$
to denote the subsequence of $\mathbf{d}$ consisting of those degrees that are larger than $n-d_1$. Note that
$\mathbf{a_f}$ is an empty sequence when $n\ge 2d_1$.

The second phase will then enumerate candidate bipartitions of $\mathbf{d}$ into $(\mathbf{a},\mathbf{b})$
by specifying which degrees will be in the left side $\mathbf{a}$, which also automatically specifies
$\mathbf{b}=\mathbf{d}-\mathbf{a}$. As we already know, we need to choose a
subsequence of $\mathbf{d}-\mathbf{a_f}$ (i.e. from those degrees in $\mathbf{d}$ that are at most $n-d_1$) with sum $S$
and concatenate $\mathbf{a_f}$ with this subsequence of degrees to form $\mathbf{a}$ based on the above discussion.
Several restrictions regarding $\ell(\mathbf{a})$ can be put on the left side $\mathbf{a}$ for the candidate bipartitions
$(\mathbf{a},\mathbf{b})$ to possibly satisfy
the Galy-Ryser conditional test so that our algorithm will enumerate as few candidate bipartitions as possible.

\begin{thm2}
	\label{Restriction_Upper1}
	The number of degrees $\ell(\mathbf{a})$ in the left side $\mathbf{a}$ cannot exceed $n-d_1$.
This is because the right side $\mathbf{b}$ contains at least $d_1$ degrees.
\end{thm2}

\begin{thm2}
	\label{Restriction_Upper2}
	Let $l_1$ be the maximum number of degrees in $\mathbf{d}$ with sum at most $\frac{|\mathbf{d}|}{2}$.
Then the number of degrees $\ell(\mathbf{a})$ in the
left side $\mathbf{a}$ cannot exceed $l_1$. This is because the degrees in the left side $\mathbf{a}$ must have
sum $\frac{|\mathbf{d}|}{2}$.
\end{thm2}

\begin{thm2}
	\label{Restriction_Lower1}
	Let $d_m$ be the minimum largest degree in any subsequence of $\mathbf{d}$ with sum at least
$\frac{|\mathbf{d}|}{2}$. Then the number of degrees $\ell(\mathbf{a})$ in the left side $\mathbf{a}$ must be at least
$d_m$. This is because the largest degree in the right side $\mathbf{b}$ must be at least $d_m$ and the conjugate
of $\mathbf{a}$ should dominate $\mathbf{b}$.
\end{thm2}

\begin{thm2}
	\label{Restriction_Lower2}
	Let $l_2$ be the minimum number of degrees in $\mathbf{d}$ with sum at least $\frac{|\mathbf{d}|}{2}$.
Then the number of degrees $\ell(\mathbf{a})$ in the
left side $\mathbf{a}$ must be at least $l_2$. The reason is similar to that for Restriction \ref{Restriction_Upper2}.
\end{thm2}

It's not hard to see that $d_m$, $l_1$ and $l_2$ can all be easily calculated with greedy algorithms.
The above discussion shows we can enumerate all subsequences $\mathbf{a}$ of $\mathbf{d}$ that
satisfies the following three requirements:
\begin{enumerate}
	\item it includes all degrees in $\mathbf{a_f}$ (i.e. those degrees in $\mathbf{d}$ that are greater than $n-d_1$).
	\item it has sum $|\mathbf{a}|=\frac{|\mathbf{d}|}{2}$.
	\item its number of degrees $\ell(\mathbf{a})$ should satisfy $\max\{d_m,l_2\}\le \ell(\mathbf{a})\le\min\{n-d_1,l_1\}$.
\end{enumerate}
%
%

In order to find a successful (i.e. satisfying the Gale-Ryser condition)
candidate bipartition $(\mathbf{a},\mathbf{b})$ of $\mathbf{d}$, our intuition is to
include a suitable number of large degrees from $\mathbf{d}-\mathbf{a_f}$ and as many small degrees of
$\mathbf{d}-\mathbf{a_f}$ as possible into $\mathbf{a}$ without violating requirement 3 mentioned above. In this way
$\mathbf{b}=\mathbf{d}-\mathbf{a}$ will not include many of the largest degrees in $\mathbf{d}$ while $\mathbf{a}$
will still include enough number of degrees, which makes it more likely for the conjugate
of $\mathbf{a}$ to dominate $\mathbf{b}$.

Following this intuition we calculate a maximum index $x_0\in\{1,2,\cdots,n\}$ such that $\mathbf{a}$ cannot include all
$\{d_1,d_2,\cdots,d_{x_0},d_{x_0+1}\}$ in order for its conjugate to dominate $\mathbf{b}$. This index $x_0$ can
be easily calculated as follows. Starting from $x=1$, if for some $x$, when
we include all $\{d_1,d_2,\cdots,d_{x}\}$ in $\mathbf{a}$ and include from $\mathbf{d}-\{d_1,d_2,\cdots,d_{x}\}$
as many smallest degrees as possible into $\mathbf{a}$ while still maintaining the correct sum
$|\mathbf{a}|=\frac{|\mathbf{d}|}{2}$, and when the number of degrees $\ell(\mathbf{a})$ in $\mathbf{a}$ starts
to fall below $\max\{d_m,l_2\}$, then $x_0$ can be chosen to be $x-1$.

After $x_0$ has been calculated, we will try to find out
if we can include a subsequence $\mathbf{d_S}$ of $\{d_1,d_2,\cdots,d_{x_0}\}$ into $\mathbf{a}$ together with some
degrees in $\mathbf{d}-\mathbf{d_S}$ such that the conjugate of $\mathbf{a}$ dominates $\mathbf{b}$. Without loss of
generality, this subsequence $\mathbf{d_S}$ can be chosen to be the largest terms of $\{d_1,d_2,\cdots,d_{x_0}\}$.
Or, equivalently, we can remove the smallest terms from $\{d_1,d_2,\cdots,d_{x_0}\}$ one at a time to get these subsequences.
For each such subsequence $\mathbf{d_S}=\{d_1,d_2,\cdots,d_{x}\}$, where $\ell(\mathbf{a_f})\le x\le x_0$ since
$\mathbf{d_S}$ necessarily includes all degrees in $\mathbf{a_f}$ according to the above discussion, we perform
the following two enumerative steps to fully construct $\mathbf{a}$:
\begin{enumerate}
\item starting from the largest possible,
choose some degree $d_y$ from $\mathbf{d}-\mathbf{d_{S'}}$ and include some copies of $d_y$ into $\mathbf{a}$.
We also stipulate that no degree larger than $d_y$ from $\mathbf{d}-\mathbf{d_{S'}}$ will be included into $\mathbf{a}$.
Here $\mathbf{d_{S'}}$ is defined as follows. If $\mathbf{d_S}$ includes all copies of $d_{x}$ from $\mathbf{d}$, then
$\mathbf{d_{S'}}$ includes $\mathbf{d_S}$ together with all copies of the degree from $\mathbf{d}$ which is immediately
smaller than $d_{x}$. If $\mathbf{d_S}$ does not include all copies of $d_{x}$ from $\mathbf{d}$, then
$\mathbf{d_{S'}}$ includes $\mathbf{d_S}$ together with all the remaining copies of $d_{x}$ from $\mathbf{d}$.
The motivation for such a definition is that we don't want $d_y$ to equal a degree we have just excluded from a previous
consideration of $\mathbf{d_S}$ when $x$ is being reduced starting from $x_0$.

\item include some small terms that are all less than $d_y$ from $\mathbf{d}-\mathbf{d_{S'}-\mathbf{d_y}}$ into $\mathbf{a}$,
where $\mathbf{d_y}$ is the subsequence of $\mathbf{d}$ consisting of all copies of $d_y$. We can generate
a number of possible combinations of small terms with each combination summing to a suitable
value based on the choice of $\mathbf{d_S}$ and the choice in the enumerative step (1) and having a suitable
number of terms so that $\ell(\mathbf{a})$ satisfies the inequality in the above requirement 3.
An appropriate procedure can be designed for this purpose such that combinations with more smaller terms are
generated first and each combination can be generated in $O(n)$ time.
\end{enumerate}

Note that both of these steps are enumerative steps. Step (1) must be exhaustive by trying each possible distinct $d_y$
from $\mathbf{d}-\mathbf{d_{S'}}$ and each of the possible number of copies up to its multiplicity in $\mathbf{d}$.
Step (2) can be non-exhaustive, which means we can impose a limit $l_c$ on the number of possible
combinations of small terms to be included into $\mathbf{a}$. This parameter $l_c$ is the place where our algorithm is
customizable and in reality we can choose $l_c$ to be a constant or a low degree polynomial of $n$.
This non-exhaustive enumeration step does open the possibility of our algorithm making an error on some ``yes'' input
instances if the specified limit $l_c$ will cause our algorithm to skip some of the possible
combinations. However, this step will not introduce any error on ``no'' input instances. We also note that some of
the choices in these two steps can be pruned during the enumerative process to speed up the enumeration phase
when they will cause $\ell(\mathbf{a})$ to fail to satisfy the inequality in the above requirement 3. In fact, the lower
bound on $\ell(\mathbf{a})$ can be improved during the process as $x$ is being reduced so that the minimum
largest degree in $\mathbf{b}$ increases.

The reader may have noticed that these enumerative steps are more sophisticated and complicated than the simple
naive scheme of enumerating all possible subsequences of $\mathbf{d}-\mathbf{a_f}$ with sum $S$. We will discuss
several alternative enumeration schemes later in Section \ref{sec:diss}.
The presented enumeration scheme here is the fastest we found through experiments.

During the enumeration phase, the algorithm will stop and output ``yes'' if a successful candidate bipartition
$(\mathbf{a},\mathbf{b})$ is found. Otherwise, it will stop enumeration and output ``no''
when the subsequence $\mathbf{d_S}$ becomes shorter than
$\mathbf{a_f}$, or, in the case that $\mathbf{a_f}$ is empty, when $\mathbf{d_S}$ includes less than half
of the largest degrees $d_1$ from $\mathbf{d}$.

We note that the enumeration phase can be easily parallelized
with respect to the different choices of $\mathbf{d_S}$. However, it may not be worth it given the good run time
performance of the serial version unless the input is long and hard (say $n=\ell(\mathbf{d})>500$).
See the following sections for run time complexity analysis and experimental evaluations.

\section{Analysis of Run Time Complexity}
\label{sec:complexity}
The seven rules in the first phase can all be checked in polynomial time. It can be easily verified that the total
running time of these rules is $O(n^3)$.

In the second phase, the three quantities $d_m$, $l_1$ and $l_2$ can all be
computed in $O(n)$ time. The maximum index $x_0$ can be calculated in $O(n^2)$ time.
The number of choices for $\mathbf{d_S}$ is $O(n)$. For each choice of $\mathbf{d_S}$, the number of choices for $d_y$
and its number of copies to be included in $\mathbf{a}$ in the enumerative step (1) is $O(n)$.
The maximum number $l_c$ of combinations of the remaining small terms to be included
in $\mathbf{a}$ in the enumerative step (2) can be chosen to be $O(1)$, $O(n)$, etc. Each combination
can be generated in $O(n)$ time.
Whenever a full left side $\mathbf{a}$ has been constructed, the
Galy-Ryser conditional test on the candidate bipartition $(\mathbf{a},\mathbf{b})$ can be performed in $O(n)$ time.
Overall, we can see that the second phase runs in $O(n^5)$ time when $l_c$ is $O(n)$.
Note this run time is achieved at the expense of the algorithm
possibly producing an erroneous output on some ``yes'' instances. However, the observed error rate is so low that we
consider the limit on $l_c$ worthwhile. On the other hand, if no limit is placed on $l_c$, then our algorithm will always
produce a correct output, at the expense of possibly running in exponential time in the worst case.

In summary, our algorithm can be customized to run in polynomial time with satisfactory low error rates (see Section
\ref{sec:exp} for some evidence of error rates).
Also note that it is a deterministic instead of a randomized algorithm.

\section{Experiments}
\label{sec:exp}
We mainly tested our implementation of the decision algorithm with the parameter $l_c$ customized as $l_c\le n$.
We first show the low error rates of the algorithm and then show the good run time performance.
\subsection{Error Rates}
We first demonstrate the somewhat surprising power of the seven rules in the first phase.
In Table \ref{tbl:results_rule_effective} we show the number
$r(n)$ of all zero-free graphical degree sequences of length $n$ that can be resolved by one of these rules and
their proportion among all $D(n)$ zero-free graphical degree sequences of length $n$.
Based on the description of the rules, these $r(n)$ instances are all ``no'' instances. The function values $r(n)$ are
obtained through a program that incorporates our decision algorithm into the algorithm to enumerate all degree
sequences of a certain length from Ruskey et al. \cite{Ruskey1994}. Let $B(n)$ be the number
of zero-free potentially bipartite graphical degree sequences of length $n$. Clearly $B(n)\le D(n)-r(n)$ since some
of the ``no'' instances are resolved in the second phase. It looks safe to conclude from this table that
$\frac{r(n)}{D(n)}$ tends to 1 as $n$ grows towards infinity and so
$\frac{B(n)}{D(n)}$ tends to 0. Note that these are just empirical observations. Rigorous proofs of the asymptotic
orders of these functions or their relative orders might require advanced techniques \cite{Wang2019}.

In fact, those instances that can be resolved by one of the seven rules are not the only ones that can avoid the
enumeration phase of our algorithm. For example, those instances that have $S=\frac{|\mathbf{d}|}{2}-|\mathbf{a_f}|=0$
can also be resolved immediately following the tests of the rules according to our description in Section \ref{sec:code}.
\begin{table}[!ht]
	\centering
	\caption{The number $D(n)$ of zero-free graphical degree sequences of length $n$, and the number $r(n)$ of them that can be resolved by one of the seven rules.}
	\begin{tabular}{||c|c|c|c||}
		\hline\hline
		$n$&$D(n)$ & $r(n)$ & $\frac{r(n)}{D(n)}$\\
		\hline\hline
		6 & 71 & 53 & 0.746479\\
		\hline
		7 & 240 & 203 & 0.845833\\
		\hline
		8 & 871 & 770 & 0.884041\\
		\hline
		9 & 3148 & 2902 & 0.921855\\
		\hline
		10 & 11655 & 10995 & 0.943372\\
		\hline
		11 & 43332 & 41603 & 0.960099 \\
		\hline
		12 & 162769 & 158074 & 0.971155 \\
		\hline
		13 & 614198 & 601556 & 0.979417 \\
		\hline
		14 & 2330537 & 2295935 & 0.985153 \\
		\hline
		15 & 8875768 & 8780992 & 0.989322 \\
		\hline
		16 & 33924859 & 33663505 & 0.992296 \\
		\hline
		17 & 130038230 & 129315300 & 0.994441 \\
		\hline
		18 & 499753855 & 497745844 &0.995982\\
		\hline
		19 & 1924912894 & 1919319963 & 0.997094\\
		\hline
		20 & 7429160296 & 7413535855 & 0.997897\\
		\hline
		21 & 28723877732 & 28680124185 & 0.998477\\
		\hline
		22 & 111236423288 & 111113621955 & 0.998896\\
		\hline
		23 & 431403470222 & 431058118392 & 0.999199\\
		\hline\hline
	\end{tabular}
	\label{tbl:results_rule_effective}
\end{table}

Next we demonstrate the low error rates of our algorithm. In Table \ref{tbl:error_rate}
we show the number $B_w(n)$ of all zero-free potentially bipartite graphical degree sequences
of length $n$ that will be incorrectly reported as a ``no'' instance if we set $l_c=1$ and their proportion
among all $B(n)$ zero-free potentially bipartite graphical degree sequences of length $n$. Even with the
smallest possible $l_c$, our algorithm makes very few errors on the ``yes'' instances. In fact, if we set $l_c=n$, then
our algorithm makes no error on all zero-free graphical degree sequences of length $n\le 23$. However,
the observed trend is that the limit $l_c$ need to grow with $n$ for our algorithm to always make
no error. We are unable to prove whether there is any polynomial of $n$ to bound $l_c$ such that our
algorithm can always give correct outputs or the error rate is always below some constant.
If $l_c$ grows faster than a polynomial of $n$, then our algorithm could run more than polynomial
time in the worst case. In our experiments we did not find any ``yes'' instance of length $n\ge 50$ that
will be misclassified by our algorithm under the setting of $l_c=n$.

\begin{table}[!ht]
	\centering
	\caption{The number $B(n)$ of zero-free potentially bipartite graphical degree sequences of length $n$, and the number $B_w(n)$ of them that will be misclassified as ``no'' instances if $l_c$ is set to 1.}
	\begin{tabular}{||c|c|c|c||}
		\hline\hline
		$n$&$B(n)$ & $B_w(n)$ & $\frac{B_w(n)}{B(n)}$ (with $l_c=1$)\\
		\hline\hline
		6 & 18 & 0 & 0 \\
		\hline
		7 & 37 & 0 & 0\\
		\hline
		8 & 100 & 0 & 0\\
		\hline
		9 & 241 & 0 & 0\\
		\hline
		10 & 640 & 0 & 0\\
		\hline
		11 & 1639 & 0 & 0 \\
		\hline
		12 & 4378 & 0 & 0 \\
		\hline
		13 & 11601 & 2 & 0.000172399 \\
		\hline
		14 & 31318 & 8 & 0.000255444 \\
		\hline
		15 & 84642 & 32 & 0.000378063 \\
		\hline
		16 & 230789 & 117 & 0.000506957 \\
		\hline
		17 & 631159 & 482 & 0.000763674 \\
		\hline
		18 & 1736329  & 1667 & 0.000960072\\
		\hline
		19 & 4790928 & 6107 & 0.0012747\\
		\hline
		20 & 13272233 & 20826 & 0.00156914\\
		\hline
		21 & 36869887 & 72879 & 0.00197665\\
		\hline
		22 & 102727688 & 244266 & 0.0023778\\
		\hline
		23 & 286893582 & 821331 & 0.00286284\\
		\hline\hline
	\end{tabular}
	\label{tbl:error_rate}
\end{table}
We note that the error rates reported in Table \ref{tbl:error_rate} is with respect to the $B(n)$
``yes'' instances. The error rate will be much lower if they are computed with respect to all $D(n)$ instances of
length $n$ because, as we know from Table \ref{tbl:results_rule_effective}, by far the majority of the ``no'' instances
have already been correctly detected by the seven rules. For example, $\frac{B_w(23)}{D(23)}=1.9\times 10^{-6}$
at the setting of $l_c=1$. Plus, increasing $l_c$ from $O(1)$ to $O(n)$ also further reduces the error rate.
For example, $\frac{B_w(23)}{D(23)}=0$ at the setting of $l_c=23$.

\subsection{Run Time Performance}
We now demonstrate the run time performance of our algorithm with the setting of $l_c=n$. Here the reported
run times were obtained through a C++ implementation tested under typical Linux workstations. We have already shown
in Section \ref{sec:complexity} that our algorithm runs in polynomial time if $l_c$ is bounded by a polynomial of $n$.
We generated random graphical degree sequences of specified length $n$, largest term $d_1$ and smallest term $d_n$.
For a wide range of $n\le 500$, we found that the hardest instances for our algorithm are approximately in the range
of $0.5n\le d_1\le 0.6n$ and $1\le d_n\le 0.1n$. The instances in these ranges are the most likely to cause our
algorithm to enter the enumeration phase. However, even the hardest instances we tested for $n\le 500$ can be finished
in about a couple of minutes, which are necessarily those ``no'' instances that will go through the entire enumeration
phase without any successful candidate bipartition being found.
All the tested instances that are decided in the first phase can be finished almost instantly. All of the tested
``yes'' instances detected in the enumeration phase can be decided in at most tens of seconds due to the empirical
fact that most of the ``yes'' instances have a successful candidate bipartition
that can be found even when $l_c$ is set to 1.

\section{Discussions}
\label{sec:diss}
We mentioned in Section \ref{sec:code} that our algorithm is customizable through the limit $l_c$ in the enumerative
step (2). In this section we describe several alternatives to the enumeration phase.

In the enumerative step (1) we have chosen $d_y$ from largest to smallest. Instead, we can choose $d_y$ from
smallest to largest. On average, we found that the former has better run time performance.

In the enumerative step (2) we prefer to enumerate the combinations of smallest terms first.
Instead, we can choose to enumerate those of
largest terms first. On average, we still found that the former has better run time performance.

The enumerative steps (1) and (2) can even be combined into one step to make the enumeration phase simpler.
That is, we can exhaustively enumerate all possible combinations of terms from $\mathbf{d}-\mathbf{d_{S'}}$ with
an appropriate sum subject to the requirement 3 about the number of terms $\ell(\mathbf{a})$ in $\mathbf{a}$.
(Or, to make it more
naive, we could exhaustively enumerate all possible combinations of terms from $\mathbf{d}-\mathbf{a_f}$ with
the sum $S$.) With these schemes we still face the choice of enumerating
largest terms first or smallest terms first. On average, the choice of ``smallest terms first'' still enjoys better run time
performance. However, in order to achieve similar low error rates in these alternative schemes with this choice
of ``smallest terms first,'' the limit on the number of combinations to be
generated will usually have to be much larger than the chosen limit
$l_c$ in our design in Section \ref{sec:code}, causing these alternatives to have much worse run time performance on
those instances that require the second phase to decide.
If no limit is placed on the number of combinations to be generated, these alternatives will all produce correct outputs
always. Nevertheless, the run time performance could become terrible. For example, for some hard instances
with length $n$ from 100 to 300, it could take days to detect a successful candidate
bipartition for ``yes'' instances and tens of days
to decide for ``no'' instances when unlimited $l_c$ is chosen, a clear evidence of exponential run time behavior.
For longer hard instances in the range $300\le n\le 500$,
these more naive enumeration phases with unlimited $l_c$ might take years or longer time to finish.

As mentioned before, our algorithm always gives the correct conclusion for ``no'' instances. But it could give an incorrect
output for some ``yes'' instances depending on the limit $l_c$ set in the enumeration phase.
This kind of behavior can be contrasted with some
randomized algorithms. The error our algorithm might make is \textit{fixed} and it comes from
the fact that not all potentially bipartite graphical degree sequences exhibit the kind of pattern that can be captured
by the particular ``limited'' search process of our algorithm. Simply put, our algorithm is deterministic.
If it makes an error on an input under a particular setting of $l_c$, it always makes an error on that input with that setting.
If a randomized algorithm makes an error on an input, then it could produce a correct output the next time it runs.

Now we comment on the complexity of the decision problem of potentially bipartiteness of graphical degree sequences.
It is obviously in $NP$. We don't know whether it is in co-$NP$ or in $P$, nor do we know whether it is $NP$-complete.
Whenever our algorithm reports an input as an ``yes'' instance, it can also output a successful
candidate bipartition. We are not sure
if this is necessary for this decision problem. For example, the well-known decision problem of primality of integers can be
decided in polynomial time \cite{AKS2004}. However, a ``composite'' output does not come with a prime factor. It is
known from the prime number theorem \cite{Hadamard1896,Poussin1896}
that almost all integers are composite. In this sense, the polynomial solvability of the primality testing problem
seems intuitive.
We would also like to compare this problem with the decision problem of whether a given graph is of class 1 or class 2, i.e.
whether its edge chromatic number is equal to $\Delta$ or $\Delta+1$ where $\Delta$ is the maximum degree of the
given graph. It is known from \cite{ERDOS1977} that almost all graphs
on $n$ vertices are of class 1 as $n$ grows towards infinity. However, it is $NP$-complete to decide whether a graph is of class 1
or class 2 \cite{Holyer1981}. These facts sound more unintuitive. It is almost certain from our experimental results
that the proportion of zero-free graphical degree
sequences of length $n$ that are not potentially bipartite approaches 1 as $n$ grows towards infinity. Is it possible that
the decision problem is actually in $P$? Or, could it be that some hidden classes of hard instances are overlooked by our
experiments and the decision problem is actually $NP$-complete or $NP$-intermediate, should $P\neq NP$.

In this paper we dealt with the decision problem of whether $\chi(\mathbf{d})=2$. In the case that $\mathbf{d}$ is not
potentially bipartite and it is desired to compute $\chi(\mathbf{d})$, we can decide, for each successive
fixed $k\ge 3$, whether there is a $k$-colorable realization of $\mathbf{d}$, until the answer becomes ``yes.''
We conjecture that each of these decision problems is $NP$-complete.

\section{Summary and directions for future research}
\label{sec:conclusion}
We presented a fast algorithm to test whether a graphical degree sequence is potentially bipartite. The algorithm
works very well in practice. It remains open whether the decision problem can be solved in polynomial time. The
complexity of the decision problem whether $\chi(\mathbf{d})\le k$ is also to be resolved.

\section{Acknowledgements}
This research has been supported by a research seed grant of Georgia Southern University. The computational
experiments have been supported by the Talon cluster of Georgia Southern University.

\sloppy

\bibliographystyle{plain}
\bibliography{pb}







\end{document}